\documentclass[10pt]{amsart}
\usepackage{graphicx} 

\title[Irreducibility and locus of complex roots of polynomials related to FLT]{Irreducibility and locus of complex roots of polynomials related to Fermat's Last Theorem}

\author{Hayk Karapetyan}
\address{Institute of Mathematics, National Academy of Sciences, Yerevan, Armenia and Faculty of Mathematics and Mechanics, Yerevan State University, Yerevan, Armenia}
\email{hayk.karapetyan6@edu.ysu.am}

\author{Ruben Hambardzumyan}
\address{Faculty of Mathematics and Mechanics, Yerevan State University, Yerevan, Armenia}

\email{rubenhambardzumyan@ysu.am}

\thanks{The work of the first author was supported by the Higher Education and Science Committee of Republic of Armenia (Research Project No 24RL-1A028).}
\thanks{The second author was supported by the Science Committee of Republic of Armenia (Research project No 23RL-1A027).}


\usepackage{amsmath, amssymb, amsthm, cancel}
\usepackage{enumitem}

\renewcommand{\le}{\leqslant}
\renewcommand{\leq}{\le}
\renewcommand{\ge}{\geqslant}
\renewcommand{\geq}{\ge}
\renewcommand{\Re}{\text{Re }}

\newtheorem{theorem}{Theorem}[section]
\newtheorem{lemma}[theorem]{Lemma}
\newtheorem{proposition}[theorem]{Proposition}
\newtheorem{corollary}[theorem]{Corollary}
\newtheorem{question}[theorem]{Question}

\theoremstyle{definition}
\newtheorem{definition}[theorem]{Definition}
\newtheorem{example}[theorem]{Example}

\theoremstyle{remark}
\newtheorem{remark}[theorem]{Remark}

\numberwithin{equation}{section}
\numberwithin{theorem}{section}

\usepackage{colonequals}

\usepackage{hyperref}
\usepackage{xcolor}
\hypersetup{
	colorlinks,
	linkcolor={red!50!black},
	citecolor={blue!50!black},
	urlcolor={blue!80!black}
}

\newcommand{\ZZ}{\mathbb{Z}}
\newcommand{\QQ}{\mathbb{Q}}
\newcommand{\FF}{\mathbb{F}}
\newcommand{\CC}{\mathbb{C}}
\newcommand{\RR}{\mathbb{R}}

\DeclareMathOperator{\cont}{cont}
\DeclareMathOperator{\disc}{disc}
\DeclareMathOperator{\Gal}{Gal}
\DeclareMathOperator{\res}{res}
\DeclareMathOperator{\im}{im}
\usepackage{tikz}

\allowdisplaybreaks

\begin{document}
	
	\begin{abstract}
		We study the polynomials \(x^n + (1-x)^n + a^n, a \in\QQ\), whose rational roots would yield counterexamples to Fermat’s Last Theorem. We investigate their factorization over $\QQ$. In the case $a \notin \{0, \pm 1\}$, we ask whether they are irreducible over $\QQ$, prove the irreducibility for several infinite families, and investigate the location of the roots of these polynomials on the complex plane. For $a=\pm1$, the factorization of $K_{a,n}$ is intimately related to that of the Cauchy--Mirimanoff polynomials $E_n$ and the polynomials $T_n$ and $S_n$ introduced by P. Nanninga. After removing the trivial factors $x$, $x-1$, and $x^2-x+1$, the remaining components agree (up to change of variable) with $E_n$, $S_n$, or $T_n$. We prove several new irreducibility results for these factors.
	\end{abstract}

    \maketitle
	
	\section{Introduction}
	
	In this article, we investigate the polynomials \(K_{a,n}(x) \colonequals x^n+(1-x)^n+a^n\), where $a \in \QQ$ and $n > 1$, which arise naturally from Fermat's Last Theorem (FLT), as the following proposition shows:
    \begin{proposition}
        The following are equivalent:
        \begin{enumerate}[label=(\roman*)]
            \item The polynomial $K_{a, n}$ has a rational root for some rational $a \neq -1$ and some integer $n > 1$.

            \item There is an $m > 2$ for which the Fermat equation $X^m + Y^m = Z^m$ has a solution in integers $X$, $Y$, and $Z$, with $XYZ \neq 0$.
        \end{enumerate}
    \end{proposition}
    \begin{proof}
         We prove (ii)$\implies$(i). There are elementary arguments showing $X^4+Y^4=Z^4$ has no nontrivial integer solutions, so $m$ must have an odd factor $n$. Then, the construction of $a$ is trivial.

         We prove (i)$\implies$(ii) by contrapositive. Suppose FLT holds. Then $K_{a, n}$ has no real roots for even $n$ since it is positive on $\RR$ and has no rational roots for odd $n > 1$ by FLT.
    \end{proof}
	
	FLT is thus equivalent to the fact that \(K_{a,n}\) does not have rational roots if \(a\in\mathbb{Q} \setminus\{-1\}\). One can consider a more general question: what does the irreducible factorization of $K_{a, n}$ over \(\mathbb{Q}\) look like? Computer calculations with SageMath show that \(K_{a,n}\) is irreducible over $\QQ$ for \(a\) and \(n\) such that \(a=\pm\frac{p}{q},a \not \in \{\pm 1, 0\}\), \(0<p,q<200\), and \(n<100\). Therefore, we ask the following question:

    \begin{question} \label{question2}
        For a rational $a \notin \{0, \pm 1\}$, are the polynomials $K_{a, n}$ irreducible over $\QQ$?
    \end{question}

    In Section \ref{sec:anotpm1}, as evidence supporting Question \ref{question2}, we prove the following:

    \begin{theorem} \label{thm:irrevid}
    For a prime $p$, let $v_p$ denote the $p$-adic valuation. Then the following statements hold:
        \begin{enumerate}[label=(\alph*), ref=\thetheorem(\alph*)]
            \item \label{thm:irrevid:part1:sqfree} If $a \neq \pm 1$, then $K_{a, n}$ is square-free for all $n$.
            \item \label{thm:irrevid:part2} If $a \neq \pm 1$, then $K_{a, n}$ does not have roots on the unit circle for all $n$. 
            \item \label{thm:irrevid:part3:v2} If $v_2(a) = -1$, then $K_{a, n}$ is irreducible over $\QQ$ for all $n$. 
            \item \label{thm:irrevid:part4} If $v_p(a) = -1$, $p$ does not divide $n$ for some prime $p$, and $n$ is odd, then $K_{a, n}$ is irreducible over $\QQ$. 
        \end{enumerate}
    \end{theorem}
    
    When \(a=0\), \(K_{a,n}\) is a modified version of the polynomial \(x^n + 1\), which has a well-known factorization. The factorization in this case is completely described in Corollary \ref{cor:factora0}. 
    
    For \(a=\pm1\), it turns out that $K_{a, n}$ may have $x$, $x - 1$, and $x^2 - x + 1$ as factors. After removing these ``trivial'' factors, the remaining polynomial seems to be irreducible. In fact, this case is strongly connected to the Cauchy--Mirimanoff polynomials \(E_n\), as well as the polynomials \(S_n\) and \(T_n\) introduced by P. Nanninga. The polynomials \(E_n\), \(S_n\), and \(T_n\) are defined by the following factorization formulae (see \cite{nanningaarticle}):
    \begin{itemize}
        \item For $n \ge 2$,
        \begin{equation} \label{eq:E}
            (x + 1)^n - x^n - 1 = x(x + 1)^a(x^2 + x + 1)^b E_n(x),
        \end{equation}
        where $a = b = 0$ if $n$ is even; while if $n$ is odd, $a = 1$ and $b = 0, 1, 2$ according as $n \equiv 3, 5, 1 \pmod{6}$.
        \item For $n \ge 1$,
        \begin{equation} \label{eq:S}
        (x + 1)^n - x^n + 1 = (x + 1)^a S_n(x),
        \end{equation}
        where $a = 0$ if $n$ is odd and $a = 1$ if $n$ is even;
        \item For $n \ge 1$, \begin{equation}\label{eq:T}
            (x + 1)^n + x^n + 1 = (x + 1)^a (x^2 + x + 1)^b T_n(x),
        \end{equation}
        where $a = 1$ and $b = 0$ if $n$ is odd; while if $n$ is even, $a = 0$ and $b = 0, 1, 2$ according as $n \equiv 0, 2, 4 \pmod{6}$.
    \end{itemize}
    The polynomials $E_n$, $S_n$, and $T_n$ are defined so that they do not have the factors \(x\), \(x+1\), or \(x^2+x+1\). 
    
    Formula \eqref{eq:E} was first noted by Cauchy for odd $n$ (see \cite{cauchy}). Later, the polynomials $E_n(x)$ were studied for prime $n\ge 11$ and conjectured to be irreducible over $\QQ$ by Mirimanoff (see \cite{Mirimanoff}). Helou noticed that most of the results of Mirimanoff were valid for all odd $n$ and suggested that the irreducibility conjecture holds for all $n \ge 2$ (see \cite{helou}). Later, Nanninga studied the polynomials $S_n$ and $T_n$, and conjectured that they are irreducible (see \cite{nanningaarticle}, \cite{Nanninga2013}). 

    The following results, related to the irreducibility of the polynomials $E_n$, $S_n$, and $T_n$, are known:
    
    \begin{itemize} 
        \item The polynomials $E_n$ and $E_m$ are coprime when $n$ and $m$ are distinct positive integers (\cite{beukers}). 

        \item For $n \ge 1$, the polynomials $E_n$ are square-free (see \cite{helou}). 
        
        \item If $n \ge 9$ is odd, then $E_n$ is reducible over $\FF_p$ for each prime $p$ (see \cite{helou}). 

        \item If $n\ge 9$ and for some prime number $p$, the polynomial $E_n$ has at most three irreducible factors over $\FF_p$, then $E_n$ is irreducible over $\QQ$ (see \cite{helou}). 

        \item For odd $n \ge 9$, the Galois group of $E_n$ over $\QQ$ is isomorphic to an extension of a subgroup of $S_3^{r_n}$ by a subgroup of $S_{r_n}$, where $r_n \colonequals \frac{\deg{E_n}}{6}$ (see \cite{helou}). 

        \item Filaseta proved that $E_{2p}$ is irreducible over $\QQ$ for all odd primes $p$. His proof is reproduced in \cite{helou}. Historically, this is the first infinite family of provably irreducible polynomials $E_n$. 

        \item For prime $p \ge 17$ such that $p \equiv 2 \pmod{3}$, each irreducible factor of $E_p$ over $\QQ$ is of degree $d \ge 12$ (see \cite{tzermias1}).

        \item For prime $p \ge 23$ such that $p \equiv 2 \pmod{3}$, the polynomial $E_p$ has an irreducible factor over $\QQ$ of degree $d \ge 18$ (see \cite{tzermias1}).

        \item Let $S$ be the set of primes greater than or equal to $19$ and congruent to $1 \pmod{3}$. There exists an effectively computable subset $S_0$ of $S$ with $|S_0| \le 6$ and such that, for any $p \in S \backslash
        S_0$, each irreducible factor of $E_p$ over $\QQ$ is of degree $d \ge 12$ (see \cite{tzermias2}). 

        \item Let $p$ be prime such that $p \equiv 1 \pmod{3}$. Suppose that there exists a prime $q \ge 11$ such that $p \equiv 1 \pmod{q}$ and $p \neq 1 \pmod{q^2}$. Then $E_p$ has an irreducible factor of degree $d \ge 6 \lfloor  \frac{q}{3} \rfloor$ over $\QQ$ (see \cite{tzermias2}). 
        
        We note that there are no infinite families of primes $r$ for which the irreducibility of $E_r$ is currently established. 

        \item In \cite{irick}, Irick proved the irreducibility of $E_{3p}$ over $\QQ$ for primes $p > 3$ and obtained a new proof of the irreducibility of $E_{2p}$ over $\QQ$ for odd primes $p$. 

        \item Suppose $p \ge 5$ is prime and $i \ge 2$. Define the polynomial $\tilde{E}_{3p^i} \in \ZZ[x]$ by 
        \[
            E_{3p^i}(x) = x^{\frac{3p^i-3}{2}} \tilde{E}_{3p^i}(x + x^{-1}).
        \]
        Then, the irreducibility of $\tilde{E}_{3p^i}$ and $E_{3p^i}$ over $\QQ$ are equivalent, and the Newton polygon of $\tilde{E}_{3p^i}(x - 2)$ with respect to $p$ has vertices 
        \begin{multline*}
            \left( \frac{p^0 - 1}{2}, i \right), 
            \left( \frac{p^1 - 1}{2}, i - 1 \right), \ldots,
            \left( \frac{p^i - 1}{2}, 0 \right)
            = \\
            \left( \frac{3p^i - 3}{2} - (p^i - 1), 0 \right), \ldots,  
            \left( \frac{3p^i - 3}{2} - (p^1 - 1), i - 1 \right), 
            \left( \frac{3p^i - 3}{2} - (p^0 - 1), i \right)
        \end{multline*}
        (see \cite{irick}).

        \item If $p \ge 5$ is prime and $i \ge 2$, then $E_{3p^i}$ is a product of at most $i$ irreducible polynomials over $\QQ$ and each of them has degree greater than or equal to $3(p - 1)$ (see \cite{irick}). 
        
        \item If $p \ge 5$ is prime and $i \ge 2$, and $P$ is an irreducible factor of $E_{3p^i}$ over $\QQ$ and $z \in \CC$ is a root of $P$, then 
        \[
             \frac{1}{z}, -z - 1, -1 - \frac{1}{z}, -\frac{1}{z + 1} - \frac{z}{z + 1}
        \]
        are roots of $P$ as well (see \cite{irick}).  

        \item In \cite{Lynch}, Lynch gave new proofs of the irreducibility of $E_{2p}$ for odd primes $p$ and $E_{3p}$ for primes $p > 3$, and proved that the polynomials $E_{5p}$ and $E_{7p}$ have at most $2$ irreducible factors over $\QQ$ for primes $p > 7$.
        
        \item For $n \ge 1$, the polynomials $S_n$ and $T_n$ are square-free (see \cite{nanningaarticle}). 
        
        \item For odd $n \ge 3$, $S_n$ and $T_n$ are irreducible over $\QQ$ (see \cite{nanningaarticle}). 
        \item For $n = 2^qm \ge 4$, where $q = 1, 2, 3, 4, 5$ and $m \ge 1$ is odd, $E_n$ and $S_n$ are irreducible over $\QQ$ (see \cite{nanningaarticle}, \cite{Nanninga2013}).
        \item For $n = 3^qm$, where $q = 1, 2, 3, 4$ and $m \ge 1$ is odd, not divisible by $3$, $E_n$ is irreducible over $\QQ$ (see \cite{Nanninga2013}).
    \end{itemize}

    Note that the formulae \eqref{eq:E}, \eqref{eq:S}, and \eqref{eq:T} imply the following identities for the polynomials $K_{a, n}$ in the case $a = \pm 1$:

    \begin{enumerate}
        \item For $a = -1$ and odd $n$,
        \[
            K_{-1,n}(x) = (1 - x)^n + x^n - 1 = x(x - 1) (x^2 - x + 1)^b E_n(-x),
        \]
        where $b = 0, 1, 2$ if $n \equiv 3, 5, 1 \pmod{6}$ respectively. 

        \item For $a = 1$ and odd $n$,
        \[
            K_{1, n}(x) = (1 - x)^n + x^n + 1 = S_n(-x).
        \]

        \item For even $n$ and $a = \pm 1$,
        \[
            K_{-1,n}(x) = K_{1,n}(x) = x^n + (1 - x)^n + 1 = (x^2 - x + 1)^b T_n(-x),
        \]
        where $b = 0, 1, 2$ according as $n \equiv 0, 2, 4 \pmod{6}$.
    \end{enumerate}
    
    Thus, this case reduces to the irreducibility of the polynomials $E_n$, $S_n$ for odd $n$ and the polynomials $T_n$ for even $n$. As mentioned above, the irreducibility of $S_n$ for odd $n$ is known (see \cite{nanningaarticle}). We will mainly focus on the subcase $T_n$ for even $n$, as evidently, it is less investigated.
    
	However, instead of working with the polynomials $E_n$ for odd $n$ and $T_n$ for even $n$ separately, we will work with the polynomials $\tilde{K}_n$, defined as follows:
    \[
        \tilde{K}_n(x) \colonequals 
        \begin{cases}
            \frac{E_n(-x)}{\cont(K_n)}, & \text{if $n$ is odd,} \\
            \frac{T_n(-x)}{\cont(K_n)}, & \text{if $n$ is even,}
        \end{cases} 
    \]
    where $\cont$ denotes the content (the \(\gcd\) of all the coefficients) of a nonzero polynomial with integer coefficients and $K_n(x) \colonequals x^n + (1 - x)^n + (-1)^n$. A formula for $\cont(K_n)$ is given in Proposition \ref{prop:cont}. 
    
    From the papers of Helou and Nanninga (see \cite{helou}, \cite{nanningaarticle}), it is evident that 
    \begin{equation} \label{eq:sym}
        \tilde K_n(x) = \tilde K_n(1 - x) =  \tilde{K}_n^*(x),
    \end{equation}
    where \(^*\) denotes the reciprocal polynomial, formed by reversing the order of the coefficients.

    A topic related to the study of Cauchy--Mirimanoff and related polynomials is the location of the roots on the complex plane. The following results are known about the location of the roots of the polynomials $E_n$ and $T_n$:

    \begin{itemize}
        \item For an odd prime $n$, exactly one third of the roots of $E_n$ lies on the unit circle, more precisely on the arc going from $-\omega$ to $-\overline{\omega}$ counterclockwise, where $\omega \colonequals e^{\frac{\pi i}{3}}$ (see \cite{Mirimanoff}).

        \item Nanninga reformulated Mirimanoff's result for all odd $n$ and noted that in that case, exactly one third of the roots of $E_n$ lies on the circular arc of the circle of radius $1$ centered at $-1$ going from $-\omega$ to $-\overline{\omega}$ clockwise; exactly one sixth of the roots of $E_n$ lies on the ray $(-\omega, -\omega+i\infty)$; and exactly one sixth of the roots of $E_n$ lies on the ray $(-\overline{\omega}, -\overline\omega-i\infty)$ (see \cite{Nanninga2013}). Furthermore, he described how to derive a similar result for the polynomials $T_n$ when $n$ is even. The result for the polynomials $E_n$ and odd $n$ has also been obtained by Lynch (see \cite{Lynch}). 
    \end{itemize}
    
    For the sake of completeness, in Section \ref{sec:loc}, we will present a complete proof of Nanninga's result for the polynomials $\tilde{K}_n$. We will need the following definition:
    
    \begin{definition}
		We will say that a sequence of polynomials \(\{P_n\}_{n\ge 1}\) \textbf{localizes} on a finite union of regular curves (i.e. curves that can be parametrized by a continuously differentiable function, with non-vanishing derivative) \( \gamma \subseteq \CC \), if the set 
		\[
			R \colonequals \{z \in \CC \mid z \text{ is a root of }P_n \text{ for some } n \in \mathbb{N}\}
		\]
		satisfies
		\(\overline R = \gamma\) (where $\overline{R}$ denotes the topological closure of $R$).
	\end{definition}

    Examples of localizing sequences of polynomials include \(x^n-1\) (on the unit circle) and Chebyshev polynomials of the first and second kind \(T_n(x)\) and \(U_n(x)\) (on the segment \([-1,1]\)).

    \begin{theorem} \label{thm:location} \leavevmode
           Call \(L\) the union of the two rays \((\omega, \omega+i\infty)\cup(\overline{\omega}, \overline\omega-i\infty)\). Call \(A_1\) the circular arc from \(\omega\) to \(\overline{\omega}\) passing through \(0\) (the center of this circle is at \(1\)). Call \(A_2\) the circular arc from \(\omega\) to \(\overline{\omega}\) passing through \(1\) (see Fig. \ref{fig}). Then the polynomials \(\tilde{K}_n\) localize on \(L\cup A_1 \cup A_2\) (the bold union of curves in Fig. \ref{fig}).
    \end{theorem}

    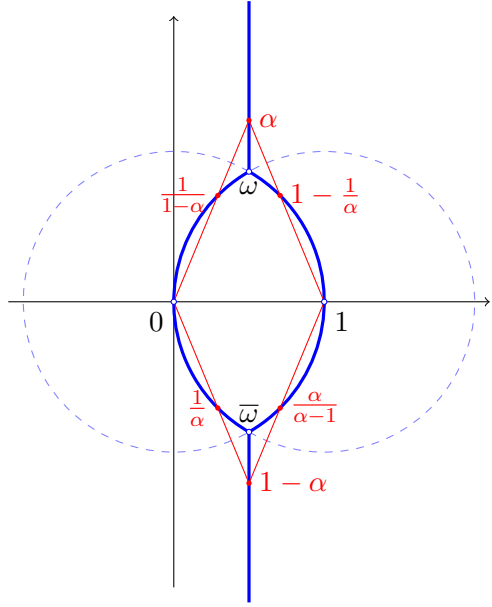
\begin{figure}[h]
        \centering
        \begin{tikzpicture}[scale = 2]
            \draw[->] (-1.1,0) -- (2.1,0);
            \draw[->] (0,-1.9) -- (0,1.9);
            
            \coordinate[label=225:$0$] (0) at (0,0);
            \coordinate[label=360-45:$1$] (1) at (1,0);
            \coordinate[label=below:$\omega$] (om) at (0.5,0.866);
            \coordinate[label=above:$\overline\omega$] (bom) at (0.5,-0.866) {};
            
            \draw[very thick,blue] (bom) arc (-60:60:1);
            \draw[very thick,blue] (om) arc (120:240:1);
            \draw[dashed, blue!50!white] (om) arc (60:300:1);
            \draw[dashed, blue!50!white] (bom) arc (-120:120:1);
            
            \draw[blue, very thick] (om) -- (0.5, 2);
            \draw[blue, very thick] (bom) -- (0.5, -2);
            
            \fill[red] (0.707, 0.707) circle (0.5pt) node[anchor=west] {$1-\frac1\alpha$};
            \fill[red] (0.293, 0.707) circle (0.5pt) node[anchor=east] {$\frac1{1-\alpha}$};
            \fill[red] (0.707, -0.707) circle (0.5pt) node[anchor=west] {$\frac\alpha{\alpha-1}$};
            \fill[red] (0.293, -0.707) circle (0.5pt) node[anchor=east] {$\frac{1}{\alpha}$};
            \fill[red] (0.5, 1.207) circle (0.5pt) node[anchor=west] {$\alpha$};
            \fill[red] (0.5, -1.207) circle (0.5pt) node[anchor=west] {$1-\alpha$};
            \draw[red] (1) -- (0.5, 1.207) -- (0) -- (0.5, -1.207) -- cycle;
            
            \filldraw[white, draw=blue] (0) circle (0.5pt);
            \filldraw[white, draw=blue] (1) circle (0.5pt);
            \filldraw[white, draw=blue] (om) circle (0.5pt);
            \filldraw[white, draw=blue] (bom) circle (0.5pt);
            
        \end{tikzpicture}
        \caption{The geometric representation of the roots of \(\tilde{K}_n\).}
		\label{fig}
	\end{figure}
    
    For even values of $n$, this theorem has also appeared in the work of Roei Raveh related to zeros of theta functions associated with even unimodular lattices (see \cite{Raveh2025}).

    In Section \ref{sec:a<1/2}, we prove a similar theorem for $K_{a, n}$, when $a$ is a fixed rational with $|a| \le \frac{1}{2}$.

    \begin{theorem} \label{easy:loc}
        If \(|a|\le\frac12\), \(K_{a,n}\) localizes on the line \(\Re x = \frac12\).
    \end{theorem}
    
    The following theorem, later used in the proofs of more specific irreducibility results, utilizes Theorem \ref{thm:location} about the location of the roots of $\tilde{K}_n$ to prove that for some \(n\), all the factors of \(\tilde{K}_n\) satisfy the equation \eqref{eq:sym} as well:
    
    \begin{theorem} \label{factorsymmetries_intro} 
	   Let \(n \geq 2\) be even, square-free, or square of a prime. Any irreducible factor \(P\in\ZZ[x]\) of \(\tilde K_n\) satisfies $P(x) = P(1 - x) = P^*(x)$. As a consequence, $6 \mid \deg{P}$. 
	\end{theorem}

    Using Theorem \ref{factorsymmetries_intro} and various variations of Eisenstein's criterion, in Section \ref{sec:6m}, we prove the irreducibility of the polynomials $\tilde{K}_{2m}$ for several infinite families of values of $m$. 
    
    Assuming the irreducibility of $\tilde{K}_{n}$, one may ask the finer question about the order or the structure of the Galois groups of these polynomials over $\QQ$. For odd $n$, it was conjectured by Helou (see \cite{helou}) that the Galois group of $E_n$ over $\QQ$ is isomorphic to the wreath product of $S_3$ by $S_{b_n}$, where $b_n \colonequals \frac{\deg \tilde{K}_n}{6}$. Computer calculations with SageMath suggest that this is also true for $\tilde{K}_n$ for even $n$. Thus, we formulate the following question:

    \begin{question}
        Is the Galois group of $\tilde{K}_{n}$ over $\QQ$ isomorphic to the wreath product of $S_3$ by $S_{b_n}$?
    \end{question}

    In Section \ref{discriminant_section}, we study the discriminant of the polynomials $\tilde{K}_n$ and prove the following theorem:

    \begin{theorem} \label{thm:oddpermut}
        If $n \ge 6$ is even and $n \not \equiv 4 \pmod{12}$, then the Galois group of $\tilde{K}_n$ over $\QQ$ contains an odd permutation. 
    \end{theorem}

    We also mention that a number of papers (see, for example, \cite{JS}, \cite{KT}, \cite{FKP}, \cite{SP}) investigate the irreducibility and Galois groups of related families of polynomials, often referred to as \textit{truncated binomial expansions}. The latter provide useful context, though they do not address our results or questions directly. 

    \section{The case $a \neq \pm 1$} \label{sec:anotpm1}
    
    
	We will denote by \(A(x,y)\) the homogenization of the univariate rational polynomial \(A\) (i.e. \(A(x,y) = y^{\deg A}A\left( \frac{x}{y}\right) \)). The following proposition will be needed for the proof of Theorem \ref{thm:irrevid:part1:sqfree}:
	\begin{proposition}\label{prop:a0} If \(d\neq2\),
    \(\Phi_d(x,1-x)\) is irreducible in \(\mathbb{Q}[x]\) with \(\deg \Phi_d(x,1-x) = \varphi(d)\). Otherwise, \(\Phi_d(x,1-x) = 1\).
	\end{proposition}
	\begin{proof}
        The statement is trivial for $d \le 2$. For $d \ge 3$, this is a special case of a well-known fact about the action of $GL_2(\QQ)$ on the set of polynomials in $\QQ[x]$ with no rational roots given by $A : f \mapsto f^{A}$, where 
        \[
            A =
            \begin{bmatrix}
                a & b\\
                c & d
            \end{bmatrix} \quad\text{and}\quad f^A(x) = (cx + d)^{\deg{f}} f\left(\frac{ax + b}{cx + d}\right)
        \] 
        for all $A \in GL_2(\QQ)$ and $f \in \QQ[x]$.
	\end{proof}

    This proposition can be used to describe the factorization of $K_{0, n}$ for $n \ge 2$ over $\QQ$. 
    
    \begin{corollary}\label{cor:factora0}
        For all $n \ge 2$, the polynomial $K_{0, n}$ factorizes into the following product of irreducible polynomials over $\QQ$:
        \[
            K_{0,n}(x) = \prod_{\substack{d \mid 2n, \ d\nmid n \\ d \neq 2}} \varPhi_d(x, 1 - x).
        \]
    \end{corollary}
    \begin{proof}
        Note that 
        \[
            K_{0, n}(x) = x^n + (1 - x)^n = \frac{x^{2n} - (1 - x)^{2n}}{x^n - (1 - x)^n} = \frac{\prod_{d \mid 2n} \varPhi_d(x, 1 - x)}{\prod_{d \mid n} \varPhi_d(x, 1 - x)} = \prod_{\substack{d \mid 2n \\ d\nmid n}} \varPhi_d(x, 1 - x).
        \]
        It remains to use Proposition \ref{prop:a0} to see that $\varPhi_d(x, 1 - x) = 1$ if $d = 2$ and is irreducible otherwise.
    \end{proof}
    
	\begin{proof}[Proof of Theorem \ref{thm:irrevid:part1:sqfree}]
		Let \[G(x) \colonequals \gcd(K_{a,n}(x), K'_{a,n}(x))=\gcd(x^n+(1-x)^n+a^n, x^{n-1}-(1-x)^{n-1}).\]
		Note that \(x^{n-1}\equiv (1-x)^{n-1} \pmod{G(x)}\), so 
        \[
        x^n+(1-x)^n+a^n \equiv x^n+(1-x)x^{n-1} + a^n \equiv x^{n-1} + a^n \pmod{G(x)}.
        \]
        Hence, \(G(x)\mid x^{n-1}+a^n\). The roots of \(G\) lie on the circle \(|x|=|a|^{\frac{n}{n-1}}\). On the other hand, they lie on the curve \(|x|=|1-x|\) (i.e. the line \(\Re x = \frac{1}{2}\)) as \(G(x)\mid x^{n-1} - (1-x)^{n-1}\). If \(|a|^{\frac{n}{n-1}}<\frac12\), the intersection of the line and the circle is empty, so \(G=1\). We cannot have \(|a|^{\frac{n}{n-1}}=\frac12\) as \(|a|\) is rational and \(\left( \frac12 \right)^{n-1} \) is not the \(n\)th power of a rational. Therefore, we consider the case when the intersection contains two points. Denote those points \(B\) and \(\overline{B}\).
		
		Note that \(G\) has real coefficients. Moreover, \(G(x) \mid x^{n-1}+a^n\) and the latter is square-free. Therefore, there are two possibilities: either \(G=1\) or \(G(x)=(x-B)(x-\overline{B})=x^2-2x\Re B + |B|^2=x^2-x+|a|^{\frac{2n}{n-1}}\).
		
		Now consider the polynomial \(x^{n-1} - (1-x)^{n-1}\). Its canonical factorization is \[
		x^{n-1} - (1-x)^{n-1} = \prod_{d\mid n-1} \Phi_d(x,1-x).
		\]
		Assume \(G\neq 1\), then \(G\) is irreducible in \(\mathbb{R}[x]\). It must be equivalent (obtained by multiplication by a nonzero constant) to some polynomial among \(\Phi_d(x,1-x)\). Since \(\deg G = 2\), \(d\in\{3,4,6\}\), meaning \(x^2-x+|a|^{\frac{2n}{n-1}}\) is equivalent to one of 
		\[
		\Phi_3(x,1-x) = x^2-x+1, \Phi_4(x,1-x) = 2x^2-2x +1, \Phi_6(x,1-x) = 3x^2-3x+1.
		\]
		However, \(|a|^{\frac{2n}{n-1}} \neq 1\) as \(|a|\neq 1\), and \(|a|^{\frac{2n}{n-1}} \in \left\{\frac12,\frac13\right\}\) is impossible as \(|a|\) is rational.
	\end{proof}

    \begin{proof}[Proof of Theorem \ref{thm:irrevid:part2}]
        Suppose $K_{a, n}$ has some root $\beta \in \CC$ that lies on the unit circle. Then $\overline{\beta} = \beta^{-1}$ is a root of $K_{a, n}$ as well. Therefore, we have 
        \begin{align*}
            K_{a, n}(\beta) &= \beta^{n} + (1 - \beta)^n + a^n = 0, \\
            \beta^n K_{a, n}\left( \frac{1}{\beta} \right) &= 1 + (\beta - 1)^n + (a\beta)^n = 0. 
        \end{align*}
        Multiplying the first equation by $(-1)^n$ and subtracting it from the second one, we obtain 
        \[
            1 - (-\beta)^n - (-a)^n + (a\beta)^n = \left( \beta^n - (-1)^n \right) \left( a^n - (-1)^n \right) = 0.
        \]
        Since $a \neq \pm 1$, it follows that $\beta^n = (-1)^n$. In particular, it follows that $\beta$ is a root of unity. Since $a^n = - \beta^n - (1 - \beta)^n$, it follows that $a^n$ is an algebraic integer. Since $a^n$ is also rational, it follows that $a \in \ZZ$. On the other hand, combining $\beta^n = (-1)^n$ with $K_{a, n}(\beta) = 0$, we obtain 
        \[
            (-1)^n + a^n = - (1 - \beta)^n.
        \]
        Therefore, 
        \[
            |a|^n - 1 \le \left|(-1)^n + a^n \right| = |1 - \beta|^n \le 2^n.
        \]
        Since $n \ge 2$, it follows that $|a| \le 2$. Since $a \neq \pm 1$, either $a = 0$ or $a = \pm 2$. If $a = 0$, then $1 - \beta$ must be a root of unity as well. In this case, \(\beta\) and \(1-\beta\) both lie on the unit circle. The only such numbers are $\omega$ and $\overline{\omega}$, and it is not difficult to see that these numbers are not roots of $K_{0, n}$. Therefore, $a = \pm 2$. If $n$ is even, then 
        \[
            2^n + 1 = \left| (-1)^n + a^n \right| = |1 - \beta|^n \le 2^n,
        \]
        which is a contradiction. Therefore, $n$ is odd, $\beta^n = -1$, and 
        \[
            a^n - 1 = (\beta - 1)^n. 
        \]
        Thus,  
        \[
            a^n = (\beta - 1)^n + 1 = \beta \left((\beta - 1)^{n - 1} - (\beta - 1)^{n - 2} + \dotsm - (\beta - 1) + 1 \right),
        \]
        and hence
        \[
            2^n = |a|^n \le |\beta - 1|^{n - 1} + |\beta - 1|^{n - 2} +\dotsm + |\beta - 1| + 1 \le 2^{n - 1} + 2^{n - 2} +  \dotsm + 2  + 1 = 2^n - 1,
        \]
        which is a contradiction. Therefore, $K_{a, n}$ does not have any roots on the unit circle for $a \neq \pm 1$. 
    \end{proof}
    
    \begin{proof}[Proof of Theorem \ref{thm:irrevid:part3:v2}]
        Suppose $a = \frac{r}{s}$, where $r \in \ZZ$, $s \in \mathbb{N}$ and $\gcd(r, s) = 1$. Since $v_2(a) = -1$, it follows that $v_2(s) = 1$ and $2 \nmid r$. Now note that 
        \[
            s^n K_{a, n}\left( \frac{x}{s} \right) = x^n + (s - x)^n + r^n. 
        \]
        Thus, the problem reduces to proving the irreducibility of the polynomial $L_n(x) \colonequals x^n + (s - x)^n + r^n$ over $\ZZ$. It suffices to show that $L_n^*$ is irreducible. The leading coefficient of $L_n^*$ is odd while all other coefficients are even, and the constant term is not divisible by $4$ (it is $2$ if $n$ is even and $sn$ if $n$ is odd). Therefore, by Eisenstein's criterion of irreducibility, $L_n^*$ is irreducible over $\QQ$. 
    \end{proof}

    \begin{proof}[Proof of Theorem \ref{thm:irrevid:part4}]
        We argue as in the proof of Theorem \ref{thm:irrevid:part3:v2}. With the same notation, the polynomial $L_n^*$ satisfies Eisenstein’s criterion at the prime $p$, since its constant term is $sn$ and $v_p(s)=1$ while $p \nmid n$. Hence $L_n^*$ is irreducible over $\mathbb{Q}$.
    \end{proof} 
    
    \section{Localization theorem for $|a| \le \frac{1}{2}$} \label{sec:a<1/2}
    In this section, we investigate the location of the roots of $K_{a, n}$ on the complex plane, for fixed $a$ with $|a| \le \frac{1}{2}$. 

    \begin{theorem}\label{maintrick}
		Fix \(a\in\mathbb{R}\). Then at least \(\left\lfloor\frac n2\right\rfloor-\left\lceil \frac{n}{\pi}\arccos\min\left( 1, \frac{1}{2|a|}\right) \right\rceil\) many roots of \(K_{a,n}\) lie in the upper half-plane on the line \(\Re x = \frac{1}{2}\), with \(|x|\geq \max \left( \frac12,|a|\right)\). Those roots form an everywhere dense set on that curve when \(n\) changes.
	\end{theorem} 
	\begin{proof}
		Denote by \(A\) the point on \(\Re x = \frac12\) in the upper half-plane with modulus \(|A|=\max \left( \frac12,|a|\right)\). Consider the variable written in the form $x=\frac{1}{2} + \frac{1}{2}i\tan\theta$, where \[
            \theta\in D\colonequals \left[ \arccos\min\left( 1,\frac1{2|a|}\right) ,\frac\pi2\right).
        \]
        Note that for the lowest value of \(\theta\), 
		
		\begin{align*}
			|x|^2& = x\overline{x} = \left(\frac{1}{2} + \frac{1}{2}i\tan\theta\right)\left(\frac{1}{2} - \frac{1}{2}i\tan\theta\right) \\
            &=\frac1{4\cos^2\theta} = \frac1{4\min\left(1,\dfrac{1}{2|a|}\right)^2 } \\
            &= \max\left(\frac{1}{2}, |a|\right)^2,
		\end{align*}
		so the map \(\theta \mapsto \frac{1}{2} + \frac{1}{2}i\tan\theta \) indeed maps \(D\) to the ray \([A,A + i\infty)\). Then,
		\begin{align*}
			K_{a,n}(x) &= \left( \frac{1}{2} + \frac{1}{2}i\tan\theta\right) ^n + \left(\frac{1}{2} - \frac{1}{2}i\tan\theta\right) ^n + a^n \\
            &=\frac{(\cos\theta+i\sin\theta)^n}{(2\cos\theta)^n} +\frac{(\cos\theta-i\sin\theta)^n}{(2\cos\theta)^n}+a^n \\
            &= \frac{2\cos n\theta}{(2\cos\theta)^n}+a^n \\
            &=\frac{2\cos n\theta + (2a\cos\theta)^n}{(2\cos\theta)^n}.
		\end{align*}
		
		Thus, it is sufficient to prove that \(f_n(\theta)\colonequals2\cos n\theta + (2a\cos\theta)^n\) has at least \(\left\lfloor\frac n2\right\rfloor-\left\lceil \frac{n}{\pi}\arccos\min\left( 1, \frac{1}{2|a|}\right) \right\rceil\) many zeros on \(D\). Consider \(f_n\) defined on \(\overline{D}\) (the topological closure). Observe that since either \(|a|\le \frac12\) or \(\theta \geq \arccos \frac1{2|a|}\), $|2a\cos \theta|\le 1$. Hence, \(f_n(\theta)\) has the same sign as $\cos n\theta$ when $\cos n\theta=\pm 1$, equivalently \(\theta = \frac{k\pi}{n},~k\in\ZZ\). The values of \(k\) for which \(\theta\in \overline{D}\) are the integers in \(\left[\frac{n}{\pi}\arccos\min\left( 1, \frac{1}{2|a|}\right), \frac{n}{2}\right]\). There are \(\left\lfloor\frac n2\right\rfloor-\left\lceil \frac{n}{\pi}\arccos\min\left( 1, \frac{1}{2|a|}\right) \right\rceil+1\) possible values for such \(k\). Since \(\cos n\theta\) has different signs for successive points \(\theta = \frac{k\pi}{n}\) and \(\theta = \frac{(k+1)\pi}{n}\), \(f_n\) has different signs as well. As \(f_n\) is continuous and real-valued, there are zeros between these successive values of \(\theta\) (these zeros are all in \(D\) as \(k=\frac{n}2\) does not yield a zero). Therefore, there are at least \(\left\lfloor\frac n2\right\rfloor-\left\lceil \frac{n}{\pi}\arccos\min\left( 1, \frac{1}{2|a|}\right) \right\rceil\) zeros of \(f_n\) (and hence of \(K_{a,n}\)). 
		
		For any interval \(\left[ \frac{k\pi}{n}, \frac{(k+1)\pi}{n}\right)\subseteq D \), there exists a root of \(f_n\) in that interval, so all the roots form an everywhere dense set in \(D\). This everywhere dense set is mapped to an everywhere dense set on the ray \([A, A+i\infty)\) by the homeomorphism \(\theta \mapsto \frac{1}{2} + \frac{1}{2}i\tan\theta \).\qedhere
	\end{proof}
    
	\begin{proof}[Proof of Theorem \ref{easy:loc}]
		Fix \(a\) and \(n\). According to Theorem \ref{maintrick}, there are at least \(\left\lfloor\frac n2\right\rfloor \) roots in the upper half-plane. Since \(K_{a,n}\) has real coefficients, the conjugates of its roots are also roots. Thus, we obtain at least \(2\left\lfloor\frac n2\right\rfloor\) roots on the line \(\Re x = \frac12\) (the conjugates are distinct from the originals as \(x=\frac12\) is not a root: it corresponds to \(\theta = 0\), and \(f_n(\theta) = 2 + (2a)^n \neq 0\)). But \(2\left\lfloor\frac n2\right\rfloor=\deg K_{a,n}\), so there are no other roots.
		
		When changing \(n\), the set of roots is everywhere dense above \(\frac12\). New roots obtained by conjugation form an everywhere dense set below \(\frac12\). Hence, the set of roots is everywhere dense on the whole line.
	\end{proof}

    \begin{corollary}
        For all $n \ge 2$ and $a \in \QQ$, such that $|a| \le \frac{1}{2}$, any factor $P$ of $K_{a, n}$ over $\QQ$ satisfies $P(1 - x) = P(x)$ and $2 \mid \deg P$. 
    \end{corollary}
    \begin{proof}
        Note that it suffices to prove this assertion only for monic irreducible factors of $K_{a, n}$ over $\QQ$. Suppose $P$ is a monic irreducible factor of $K_{a, n}$ over $\QQ$. From Theorem \ref{easy:loc}, any root $\alpha$ of $P$ is non-real and satisfies $\overline{\alpha} = 1 - \alpha$. Thus, $\deg{P}$ is even and the minimal polynomial of $\overline{\alpha} = 1 - \alpha$ over $\QQ$ is $P(1 - x)$. Since $P(\overline{\alpha}) = 0$ and both $P(x)$ and $P(1 - x)$ are irreducible, it follows that $P(x) = P(1 - x)$.
    \end{proof}

    \section{Preliminary results for $a = \pm 1$} \label{sec:loc}
	Now we move on to investigating the case \(a = \pm 1\). Recall that $K_n \colonequals K_{-1, n}$. The following proposition is a mild generalization of Nanninga's results (see \cite{Nanninga2013}) and is included here for the sake of completeness.
    
    \begin{proposition} \label{prop:cont}
        We have
        \[
            \cont (K_n) = 
            \begin{cases}
                1, & \text{if $n$ is not a power of a prime}, \\
                p, & \text{if $n = p^e$, where $p$ is a prime and $e \ge 1$.}
            \end{cases}
        \]
    \end{proposition}
    \begin{proof}
        We will first find all primes \(p\) with \(p\mid \cont (K_n)\). Let \(p\) be such a prime and \(n=p^em\), where \(\gcd(p,m)=1\) and \(e\) is a nonnegative integer. We have \(K_n = 0\) in \(\FF_p\). However,
        \[
        x^n+(1-x)^n+(-1)^n=(x^m+(1-x)^m+(-1)^m)^{p^e}
        \]
        in \(\FF_p\), so \(K_n = 0 \) is equivalent to \(x^m+(1-x)^m+(-1)^m = 0\). If \(m=1\), this is indeed true. Otherwise, the coefficient of the linear term of this latter polynomial equals \(-m\), implying \(p\mid m\), a contradiction. Hence, \(n = p^e\).

        Thus, if \(n\) is not a power of a prime, \(\cont (K_n) = 1\). Otherwise, if \(n=p^e\), the only prime dividing \(\cont (K_n)\) is \(p\). To prove that \(\cont (K_n)=p\), we will find a coefficient with \(p\)-adic valuation \(1\). Consider the coefficient of \(x^{p^{e-1}}\) in \(K_n\). Since \(0<p^{e-1}<p^e\), it equals \((-1)^{p^{e-1}}\binom{p^e}{p^{e-1}}\). According to Legendre's formula,
        \begin{align*}
           v_p\left(\binom{p^e}{p^{e-1}}\right) &= v_p((p^e)!) - v_p((p^e-p^{e-1})!) - v_p((p^{e-1})!)\\&= \sum_{i=0}^e \left(\left\lfloor \frac{p^e}{p^i}\right\rfloor - \left\lfloor \frac{p^e-p^{e-1}}{p^i}\right\rfloor - \left\lfloor \frac{p^{e-1}}{p^i}\right\rfloor\right)\\&= \underbrace{0+0+\dotsm+0}_{i\leq e-1}+1 = 1. \qedhere
        \end{align*}
    \end{proof}

    Denote \(d_n \colonequals \deg \tilde{K}_n\). From the definition of the polynomials $E_n$, $T_n$ (see \cite{nanningaarticle}) the following formula holds:
    \begin{equation*}
	d_n = \begin{cases}
		n-7, &\text{if } n\equiv 1\pmod 6,\\
		6\left\lfloor\frac{n}{6}\right\rfloor, &\text{otherwise}.
	\end{cases}
	\end{equation*}

    \begin{remark} \label{constant_values}
        \(\tilde K_n\) is constant if and only if \(n = 2,3,4,5,7\).
    \end{remark}

    Here are the \(\tilde K_n\) for \(2\le n\le 15\):
    \begin{align*}
        \tilde{K}_{2}(x) &= \tilde{K}_3(x) = \tilde{K}_4(x) = \tilde{K}_5(x) = \tilde{K}_7(x) = 1\\
        \tilde{K}_{6}(x)&=2x^6 - 6x^5 + 15x^4 - 20x^3 + 15x^2 - 6x + 2\\
        \tilde{K}_{8}(x)&=x^6 - 3x^5 + 10x^4 - 15x^3 + 10x^2 - 3x + 1\\
        \tilde{K}_{9}(x)&=3x^6 - 9x^5 + 19x^4 - 23x^3 + 19x^2 - 9x + 3\\
        \tilde{K}_{10}(x)&=2x^6 - 6x^5 + 27x^4 - 44x^3 + 27x^2 - 6x + 2\\
        \tilde{K}_{11}(x)&=x^6 - 3x^5 + 7x^4 - 9x^3 + 7x^2 - 3x + 1\\
        \tilde{K}_{12}(x)&=2x^{12} - 12x^{11} + 66x^{10} - 220x^9 + 495x^8 - 792x^7 + 924x^6 \\ &- 792x^5+ 495x^4 - 220x^3 + 66x^2 - 12x + 2\\
        \tilde{K}_{13}(x)&=x^6 - 3x^5 + 8x^4 - 11x^3 + 8x^2 - 3x + 1\\
        \tilde{K}_{14}(x)&=2x^{12} - 12x^{11} + 77x^{10} - 275x^9 + 649x^8 - 1078x^7 + 1276x^6\\ & - 1078x^5 + 649x^4 - 275x^3 + 77x^2 - 12x + 2\\
        \tilde{K}_{15}(x)&=15x^{12} - 90x^{11} + 365x^{10} - 1000x^9 + 2003x^8 - 3002x^7 + 3433x^6\\ & - 3002x^5 + 2003x^4 - 1000x^3 + 365x^2 - 90x + 15
    \end{align*}

    In \cite{nanningaarticle}, Nanninga proved that the polynomials $E_n$ and $T_n$ are square-free, and hence the polynomials $\tilde{K}_n$ are square-free as well. The polynomials $E_n$ and $T_n$ are defined so that $-\omega$, $-\overline{\omega}$, $-1$, and $0$ are not roots of them. Thus, neither of the polynomials $\tilde{K}_n$ has a root at $\omega$, $\overline{\omega}$, $1$, or $0$.

	\begin{proof}[Proof of Theorem \ref{thm:location}]
		Theorem \ref{maintrick} implies (after taking conjugates of the roots) that there are at least \(2\left( \left\lfloor\frac{n}{2}\right\rfloor - \left\lceil\frac{n}{3}\right\rceil\right) \) many roots of $\tilde{K}_n$ on \(L\). Straightforward checking shows \( \left\lfloor\frac{n}{2}\right\rfloor - \left\lceil\frac{n}{3}\right\rceil = \frac{d_n}{6}\). Now observe that 
		\begin{equation}\label{eq:symkn}
		    K_n(x)=K_n(1-x)=(-x)^{n}K_n\left( \frac1{x}\right). 
		\end{equation}
		This implies that the roots of \(\tilde{K}_n\) are mapped to other roots under the maps $x\mapsto 1-x$ and $x\mapsto \frac1x$. The map \(x\mapsto \frac1x\) is a geometric inversion with center \(0\) and radius \(1\), followed by a reflection across the real axis. The inversion of the line \(\Re x=\frac12\) is a circle passing through zero. It should also pass through \(\omega\) and \(\overline\omega\) since those remain fixed (see \cite[pp. 77-83]{Cox} for the basic theory of geometric inversion). Thus, \(x\mapsto \frac1x\) maps the \(\frac{d_n}{3}\) roots on \(L\) to \(A_1\). Then the map \(x\mapsto 1-x\) (a central symmetry across \(\frac{1}{2}\)) maps those roots to \(A_2\). In total, we have found \(d_n\) many roots on the curves, so there can be no more.
		
		Since the maps $x\mapsto \frac1x$ and $x\mapsto 1-x$ are homeomorphisms on \(L\) and \(A_1\) respectively, they map the set of everywhere dense roots on \(L\) to a set of everywhere dense roots on \(A_2\) and \(A_1\) when \(n\) changes. \qedhere
	\end{proof}

    \begin{proposition}\label{prop:cyc}
        \(\tilde K_n\) is coprime to all cyclotomic polynomials. Equivalently, all the roots on the right arc have an irrational argument (with respect to \(2\pi\)).
    \end{proposition}
    \begin{proof}
        Since this is equivalent to showing that the polynomials $\tilde{K}_n$ do not vanish on the set of the roots of unity, which is closed under taking additive inverses, it suffices to check that \(\gcd(\tilde{K}_n(x), \varPhi_d(-x)) = 1\) for all positive integers \(n, d\). Denote \(\zeta_d=e^{\frac{2\pi i}{d}}\). Since $\tilde{K}_n(x) = 1$ for $n \le 5$, we may assume that $n \ge 6$. Since $-1$ is not a root of $K_n$ and $1, \omega, \overline{\omega}$ are not roots of $\tilde{K}_n$, we may assume that $d \ge 4$. Then 
        \[
            |1 + \zeta_d|^n = \left| 2\cos{\left(\frac{\pi}{d}\right)} \right|^n > 2 \geq |\zeta_d^n + (-1)^n|,
        \]
        implying $K_n(-\zeta_d) \neq 0$. This means \(\tilde{K}_n(-\zeta_d)\neq 0\), so \(\gcd(\tilde{K}_n(x), \varPhi_d(-x)) = 1\).
    \end{proof}

    \begin{definition}
        Denote by \(H\) the group of transformations \[
        \left\{x\mapsto x, x \mapsto 1 - x, x \mapsto \frac{1}{x}, x \mapsto \frac{x}{x - 1}, x \mapsto \frac{1}{1 - x}, x \mapsto \frac{x - 1}{x}\right\}.
        \]
    \end{definition}
    As linear rational functions, they can be represented as matrices in \(PGL_2(\mathbb{C})\). It is not difficult to verify by matrix multiplication that \(H\) is a group of order \(6\) which is not abelian, so it is isomorphic to the symmetric group \(S_3\). Note that $H$ acts on $\CC[x]$, by the transformation \( \left( x\mapsto\frac{ax+b}{cx+d} \right) \in H\), sending a polynomial $f(x) \in \CC[x]$ to $(cx + d)^{\deg{f}} f \left( \frac{ax + b}{cx + d}\right)$. By a combination of remarks of Helou and Nanninga (see \cite{helou}, \cite{nanningaarticle}), $\tilde{K}_n$ is fixed by all of the elements of $H$.

    From a result of Helou (see \cite{helou}), it follows that all the factors of $\tilde{K}_n$ over $\QQ$ are invariant under the action of $H$ for odd primes $n$. Now we prove Theorem \ref{factorsymmetries_intro}, which gives the same result for the case of $n$ being even, square-free, or square of a prime.
    
	\begin{proof}[Proof of Theorem \ref{factorsymmetries_intro}]
		It suffices to prove that for every irreducible factor $P$ of $\tilde{K}_n$ and every root $r$ of $P$, the symmetric copies of $r$ (images under the transformations of $H$) are also roots of $P$.
        
        Note that for such \(n\), $\tilde{K}_n(0)$ is square-free. It is sufficient to prove the theorem for the case when \(P\) contains a root on the arc $A_2$. The main statement will then follow by the following reasoning: any irreducible factor \(Q\) of \(\tilde K_n\) has a root \(r'\) which has a symmetric copy \(r''\) on $A_2$. The minimal polynomial \(R\) of \(r''\) has all the symmetric copies of \(r''\), including \(r'\) as roots. Since \(\tilde K_n\) is square-free, \(Q\) has to be \(R\), so it satisfies the required property.
		
		Now let \(P\) be an irreducible factor of \(\tilde K_n\) which has a root \(r\) on $A_2$. \(P\) automatically has content \(1\) as \(\cont (\tilde K_n) = 1\). Since \(P\) has real coefficients, \(\overline r = \frac1r\) is also a root of \(P\). Polynomials \(P\) and \(P^*\) have a common root, so they are not relatively prime. Since \(P\) is irreducible, we get \(P\mid P^*\). But \(P\) and \(P^*\) have the same degree, so \(P^* = cP\) for some constant \(c\). Note that \(P^*(1) = 1^{\deg P}P(\frac11) = P(1)\), and hence \(c=1\).
		
		Next, observe that if \(P\) has a root on $L$, similar reasoning yields \(P(x)=P(1-x)\) (to prove that \(c=1\), we will need to plug in \(x=\frac12\)). In that case, \(P\) has two symmetries generating \(S_3\), so it satisfies the required condition. Similar reasoning applies if \(P\) has a root on the left arc (this time with the symmetry \(x\mapsto \frac{x}{x-1}\)). Thus, we may assume that \(P\) has roots only on $A_2$. We now proceed to showing that this case is impossible. 
        
        Denote \(Q(x)=P(1-x)\). Note that \(P, Q, Q^*\) are all irreducible primitive polynomials with disjoint set of roots (they lie on $A_2$, $A_1$, $L$ respectively), so 
        \[
		    P(x)Q(x)Q^*(x)\mid \tilde K_n(x).
		\]
        
        Plug in \(x=0\) to get \(P(0)P(1)l \mid \tilde K_n(0)\), where \(l\) is the leading coefficient of \(P(1-x)\). Since \(P\) has no real roots, its degree is even, so \(l\) is also the leading coefficient of \(P(x)\). Considering the fact that \(P=P^*\), we have \(l=P(0)\), implying \(P(0)^2P(1)\mid \tilde K_n(0)\). Since \(\tilde K_n(0)\) is square-free, we get \(P(0)=\pm 1\). Therefore, $P$ is a monic irreducible polynomial with integer coefficients with all roots lying on the unit circle. By Kronecker's theorem (see \cite{cyclo}), $P$ is a cyclotomic polynomial. On the other hand, from Proposition \ref{prop:cyc}, $\tilde{K}_n$ is coprime to all cyclotomic polynomials, which is a contradiction. 

        Hence, \(P\) also satisfies \(P(x) = P(1-x) = P^*(x)\).
	\end{proof}

    The following theorem investigates the reducibility of the polynomials \(f\) satisfying $f(x)=f^*(x) = f(1 - x)$ modulo primes. A similar theorem was proved by Helou (see \cite{helou}), but our proof is significantly different. 

    \begin{theorem} \label{prime_reductions}
        Let $f \in \ZZ[x]$ be a non-constant polynomial satisfying $f(1 - x) = f(x) = f^*(x)$ and let $p$ be a prime number. Denote by $\overline{f}$ the reduction of $f$ modulo $p$. Then one of the following is true:
        \begin{itemize}
            \item $\overline{f} = 0$. 
            \item $f(x) = c(x^2 - x + 1)$ for some $c$ with $p \nmid c$ and $p \equiv 2 \pmod{3}$.  
            \item $\overline{f}$ is reducible in $\FF_p[x]$.
        \end{itemize} 
    \end{theorem}
    \begin{proof}
        Suppose $\overline{f}$ is irreducible in $\FF_p[x]$. Note that $\overline{f}(1 - x) = \overline{f}(x)$. If the leading coefficient of \(f\) is divisible by $p$, then the constant term of $f$ is divisible by $p$ as well. Then $0$ is a root of $\overline{f}$ in $\FF_p$, and hence either $\overline{f} = 0$ or $\overline{f}(x) = cx$ for some $c \in \FF_p^\times$. In the latter case, \(\overline{f}(1 - x) = \overline{f}(x)\) implies $c(1 - x) = cx$, and hence $c = 0$. Therefore, we may assume that the leading coefficient of $f$ is not divisible by $p$, which implies \(\deg\overline{f} = \deg{f}\), the constant term of \(\overline{f}\) is different from \(0\), and $\overline{f}^*(x) = \overline{f^*}(x) = \overline{f}(1 - x) = \overline{f}(x)$.
        
        Let $L$ be the splitting field of $\overline{f}$ over $\FF_p$, and let $G$ denote the Galois group $\Gal(L/\FF_p)$. Fix a root $\alpha \in L$ of $\overline{f}$. Since $0$ is not a root of $\overline{f}$, $\alpha \neq 0$. Note that by $\overline{f}(1 - x) = \overline{f}(x) = \overline{f}^*(x)$, it follows that $1 - \alpha$ and $\alpha^{-1}$ are roots of $\overline{f}$ as well. Since $\overline{f} \in \FF_p[x]$ and $\overline{f}$ is irreducible, $\{\alpha, \alpha^p, \alpha^{p^2}, \ldots, \alpha^{p^{\deg \overline{f} - 1}}\}$ is the set of roots of $\overline{f}$ (see \cite{DF}). Thus, for some index $j$, 
        \[
            \alpha^{p^j} = 1 - \alpha.
        \]
        On the other hand, since $\overline{f}$ is irreducible in $\FF_p[x]$, $G$ acts transitively on the set of roots of $\overline{f}$ (see \cite{DF}). Since $\alpha^{-1}$ is a root of $\overline{f}$, there is an automorphism $\sigma \in G$ such that $\sigma(\alpha) = \alpha^{-1}$. Applying $\sigma$ to $\alpha^{p^j} = 1 - \alpha$, we see that 
        \[
            \alpha^{-p^j} = 1 - \alpha^{-1}. 
        \]
        Thus, $(1 - \alpha)^{-1} = 1 - \alpha^{-1}$, implying $\alpha^2 - \alpha + 1 = 0$. Thus, $\overline{f}(x)$ and $x^2 - x + 1$ have a common factor over $L$. Hence, they must have a common factor over $\FF_p$. Since $\overline{f}$ is irreducible in $\FF_p[x]$, it follows that $\overline{f}(x) \mid x^2 - x + 1$ in $\FF_p[x]$. Thus, $\deg{f} = \deg{\overline{f}} \in \{0, 1, 2\}$. Since $f$ is non-constant, either $\deg{f} = 1$ or $\deg{f} = 2$.
        \begin{itemize}
            \item If $\deg{f} = 1$, then $f(x) = ax + b$ for some $a, b \in \ZZ$ with $a \neq 0$. From $f(x) = f^*(x)$, we must have $a = b$. Then $f (1 - x) = a(1 - x) + a = - ax + 2a = f(x) = ax + a$. Thus, $a = 0$ and $f = 0$, which is a contradiction. 
            
            \item If $\deg{f} = 2$, then $f(x) = ax^2 + bx + c$ for some $a, b, c \in \ZZ$ with $a \neq 0$. From $f = f^*$ it follows $a = c$. On the other hand, $f(1 - x) = f(x)$ implies
            \begin{align*}
                cx^2 + bx + c &= c(1 - x)^2 + b(1 - x) + c \\
                &= cx^2 - 2cx + c + b - bx + c \\
                &= cx^2 - (b + 2c)x + b + 2c.
            \end{align*}  
            Therefore, $b = -c$, and $f(x) = c(x^2 - x + 1)$. Since $\deg{\overline{f}} = 2$, $p \nmid c$. It remains to show that $p \equiv 2 \pmod{3}$. If $p = 2$, there is nothing to prove, so assume $p$ is odd. From the irreducibility of $\overline{f}$ in $\FF_p[x]$, it follows that $x^2 - x + 1$ is irreducible over $\FF_p$. This fact is equivalent to \(4x^2-4x+4=(2x-1)^2+3\) not having a root, that is, \(p\neq 3\) and the Legendre symbol \(\left(\frac{-3}p\right)=-1\). By quadratic reciprocity, this is equivalent to \(\left(\frac p{3}\right)=-1\), which happens if and only if $p \equiv 2 \pmod{3}$.\qedhere 
        \end{itemize}
    \end{proof}

    The following corollary suggests that the investigation of the irreducibility of polynomials $\tilde{K}_n$ might be difficult.
    
    \begin{corollary}
        $\tilde K_{n}$ is constant for $n = 2, 3, 4, 5, 7$ and is reducible modulo each prime $p$ otherwise. 
    \end{corollary}
    \begin{proof}
        By Remark \ref{constant_values}, $\tilde{K}_n$ is constant if and only if $n = 2, 3, 4, 5, 7$. Otherwise, $\tilde{K}_n$ is non-constant and primitive, hence, by Theorem \ref{prime_reductions}, either $\tilde{K}_n$ is reducible modulo $p$ or $\tilde{K}_n(x) = c(x^2 - x + 1)$ for some nonzero $c \in \ZZ$. From the definition of $\tilde{K}_n(x)$, it is coprime with $x^2 - x + 1$ over $\QQ$. Thus, the second case is impossible, and $\tilde{K}_n$ is reducible modulo $p$. 
    \end{proof}

    \begin{corollary} \label{cor3.10}
        Suppose $n$ is even, square-free, or square of a prime and $n \neq 2, 3, 4, 5, 7$, each irreducible factor of $\tilde{K}_n$ over $\ZZ$ is reducible modulo each prime $p$. 
    \end{corollary}
    \begin{proof}
        By Theorem \ref{factorsymmetries_intro}, each irreducible factor $f$ of $\tilde{K}_n$ satisfies $f(1 - x) = f(x) = f^*(x)$. On the other hand, since $\tilde{K}_n$ is primitive and is coprime to $x^2 - x + 1$, $f$ is irreducible and coprime to $x^2 - x + 1$ as well. Thus, $f$ is reducible modulo each prime $p$ by Theorem \ref{prime_reductions}. 
    \end{proof}

    The following corollary is a generalization of Helou's result from the case of Cauchy--Mirimanoff polynomials of odd prime index, to a larger class of values of $n$:

    \begin{corollary}
        Suppose $n$ is even, square-free, or square of a prime and $n \neq 2, 3, 4, 5, 7$. If there is a prime $p$ such that $\tilde{K}_n$ factorizes into at most $3$ irreducible factors over $\FF_p$, then $\tilde{K}_n$ is irreducible over $\QQ$.
    \end{corollary}
    \begin{proof}
        This directly follows from Corollary \ref{cor3.10}.
    \end{proof}
	
    \section{Irreducibility results for \(\tilde{K}_{2m}\)} \label{sec:6m}
    In this section, we will prove the irreducibility of some of the polynomials $\tilde{K}_{2m}$. First, we will prove the following two propositions that give the irreducibility of $\tilde{K}_{3^a + 3^b}$ and $\tilde{K}_{2 \cdot 3^a}$, for $a, b \ge 1$ and then move on to the irreducibility of $\tilde{K}_{2lp}$ for $l \in \{1, 2, 3\}$ and $p$ prime.

    \begin{proposition}\label{p=3}
        For $a, b \ge 1$, $\tilde{K}_{3^a + 3^b}$ is irreducible over $\QQ$. 
    \end{proposition}
    \begin{proof}
        Note that $\tilde{K}_{3^a + 3^b} = K_{3^a + 3^b}$. In \(\FF_3\), we have the following equality:
        \begin{align*}
            K_{3^a + 3^b}(x) &= x^{3^a+3^b} + (1-x)^{3^a+3^b} + 1 \\
            &= x^{3^a+3^b}+(1-x)^{3^a}(1-x)^{3^b} + 1 \\
            &= x^{3^a+3^b} + (1-x^{3^a})(1-x^{3^b})+1 \\
            &= x^{3^a+3^b}+1-x^{3^a}-x^{3^b}+x^{3^a+3^b}+1 \\ 
            &= 2x^{3^a+3^b} +2x^{3^a}+2x^{3^b}+2 \\ 
            &= 2(x^{3^a}+1)(x^{3^b}+1) \\ 
            &= -(x+1)^{3^a}(x+1)^{3^b} \\
            &=-(x+1)^{3^a+3^b}.
        \end{align*}
        Therefore, Eisenstein's criterion of irreducibility is applicable to \(K_{3^a + 3^b}(x-1)\). Since the constant term of \(K_{3^a + 3^b}(x-1)\) equals 
        \[ K_{3^a + 3^b}(-1) = 2^{3^a + 3^b} + 2 = 64^{(3^a + 3^b)/6} + 2 \equiv 3 \pmod 9, \]
        Eisenstein's criterion concludes the proof.
    \end{proof}

    \begin{proposition}\label{p=2}
        For $a \ge 1$, $\tilde{K}_{3 \cdot 2^a}$ is irreducible over $\QQ$.
    \end{proposition}
    \begin{proof}
        Note that $\tilde{K}_{3 \cdot 2^a} = K_{3 \cdot 2^a}$. Assume \(K_{3\cdot 2^a}(x) = A(x)B(x)\). Note that \(K_{3\cdot 2^a}(0) = 2\), so we can assume \(A(0) = 1, B(0) = 2\) without loss of generality. According to Theorem \ref{factorsymmetries_intro}, \(A(x) = A(1-x) = A^*(x)\) and \(B(x) = B(1-x)=B^*(x)\), so we also have \(A(1) = 1, B(1) = 2\). Over \(\FF_2\),
        \begin{align*}
            K_{3 \cdot 2^a}(x) &= x^{3\cdot 2^a} + (1 - x)^{3 \cdot 2^a} + 1 \\
            &= \left( x^3 + (1 - x)^3 + 1\right)^{2^a} \\
            &= \left( x^2 + x \right)^{2^a} \\ 
            &= x^{2^a}(x + 1)^{2^a}.
        \end{align*}
        Since $A(0) = A(1) = 1$, $A(x)$ is coprime with $x(x + 1)$ and divides \((x^2+x)^{2^a}\) in $\FF_2[x]$. Thus, $A = 1$ modulo $2$. Since $A^* = A$, the leading coefficient and the constant term of $A$ are equal. Since $A = 1$ over $\FF_2$, it follows that $A$ is constant in $\ZZ[x]$. Thus, $K_{3\cdot 2^a}$ is irreducible over $\QQ$. 
    \end{proof} 

    \begin{theorem}
        Let $p$ be an odd prime. Then $\tilde{K}_{2p}$ is irreducible over $\QQ$.
    \end{theorem}
    \begin{proof}
        The case $p = 3$ follows from Proposition \ref{p=3}, so we may assume $p > 3$. Note that the leading coefficient of $\tilde{K}_{2p}$ is $2$, and $\tilde{K}_{2p}(x)$ divides $K_{2p}(x)$ over $\ZZ$. Since, over $\FF_p$, 
        \[
            K_{2p}(x) = x^{2p} + (1 - x)^{2p} + 1 = (x^2 + (1 - x)^2 + 1)^p = 2(x^2 - x + 1)^p
        \]
        and $\tilde{K}_{2p}(x) = \tilde{K}_{2p}(1 - x) = \tilde{K}_{2p}^*(x)$, $\tilde{K}_{2p}(x) = 2(x^2 - x + 1)^{d_{2p}/2}$ over $\FF_p$. By Theorem \ref{factorsymmetries_intro}, each factor $P$ of $\tilde{K}_{2p}$ over $\ZZ$ satisfies $P(x) = P(1 - x) = P^*(x)$, and hence satisfies the same equation over $\FF_p$. Therefore, each factor of $\tilde{K}_{2p}$ over $\FF_p$ has the form $c(x^2 - x + 1)^j$, where $c \in \FF_p^\times$ and $1 \le j \le \frac{d_{2p}}{2}$. Therefore, if $\tilde{K}_{2p}$ is reducible over $\QQ$, then $\tilde{K}_{2p}(\omega)$ has to be divisible by $p^2$ over $\ZZ$. We will verify by direct calculations that this is not the case, proving that $\tilde{K}_{2p}$ is irreducible over $\QQ$. 
        
        If $p \equiv 1 \pmod{3}$, then 
        \[
            \tilde{K}_{2p}(x) = T_{2p}(-x) = \frac{x^{2p} + (1 - x)^{2p} + 1}{x^2 - x + 1} = \frac{K_{2p}(x)}{x^2 - x + 1},
        \]
        and the Taylor expansion of $K_{2p}$ at $\omega$ has the form \(K'_{2p}(\omega)(x - \omega) + (x - \omega)^2Q(x)\) for some $Q \in \CC[x]$. Therefore, 
        \begin{align*}
            \tilde{K}_{2p}(x) &= \frac{K'_{2p}(\omega)+ (x - \omega)Q(x)}{x - 1 + \omega} \\
            &= \frac{2p(\omega^{2p - 1} - (1 - \omega)^{2p - 1}) + (x - \omega)Q(x)}{x - 1 + \omega} \\
            &= \frac{2p(2\omega - 1)+ (x - \omega)Q(x)}{x - 1 + \omega},
        \end{align*}
        and hence $\tilde{K}_{2p}(\omega) = 2p$. If $p \equiv 2 \pmod{3}$, then 
        \[
            \tilde{K}_{2p}(x) = T_{2p}(-x) = \frac{x^{2p} + (1 - x)^{2p} + 1}{(x^2 - x + 1)^2} = \frac{K_{2p}(x)}{(x^2 - x + 1)^2},
        \]
        and the Taylor expansion of $K_{2p}$ at $\omega$ has the form $\frac{K''_{2p}(\omega)}{2}(x - \omega)^2 + (x - \omega)^3Q(x)$ for some $Q \in \CC[x]$. Therefore, 
        \begin{align*}
            \tilde{K}_{2p}(x) &= \frac{K''_{2p}(\omega)+ 2(x - \omega)Q(x)}{2(x - 1 + \omega)^2} \\
            &= \frac{2p(2p - 1)(\omega^{2p - 2} + (1 - \omega)^{2p -2}) + 2(x - \omega)Q(x)}{2(x - 1 + \omega)^2} \\
            &= \frac{p(2 p - 1)(\omega^2 + (1 - \omega)^2)+ (x - \omega)Q(x)}{(x - 1 + \omega)^2} \\
            &= \frac{-p(2p - 1) + (x - \omega)Q(x)}{(x - 1 + \omega)^2},
        \end{align*}
        and hence $\tilde{K}_{2p}(\omega) = -\frac{p(2p - 1)}{(2\omega - 1)^2} = \frac{p(2p - 1)}{3}$. In either case, $p^2 \nmid \tilde{K}_{2p}(\omega)$, as desired. 
    \end{proof}

    \begin{theorem}
        Let $p$ be a prime. Then $\tilde{K}_{4p}$ is irreducible over $\QQ$. 
    \end{theorem}
    \begin{proof}
        Since $d_8 = 6$, Theorem \ref{factorsymmetries_intro} implies the irreducibility of $\tilde{K}_8$ over $\QQ$. Proposition \ref{p=3} implies the irreducibility of $\tilde{K}_{12}$ over $\QQ$. Hence, we may assume $p > 3$. Note that the leading coefficient of $\tilde{K}_{4p}$ is $2$, and $\tilde{K}_{4p}(x)$ divides $K_{4p}(x)$ over $\ZZ$. Since, over $\FF_p$, 
        \[
            K_{4p}(x) = x^{4p} + (1 - x)^{4p} + 1 = (x^4 + (1 - x)^4 + 1)^p = 2(x^2 - x + 1)^{2p}
        \]
        and $\tilde{K}_{4p}(x) = \tilde{K}_{4p}(1 - x) = \tilde{K}_{4p}^*(x)$, $\tilde{K}_{4p}(x) = 2(x^2 - x + 1)^{d_{4p}/2}$ over $\FF_p$. By Theorem \ref{factorsymmetries_intro}, each factor $P$ of $\tilde{K}_{4p}$ over $\ZZ$ satisfies $P(x) = P(1 - x) = P^*(x)$, and hence satisfies the same equation over $\FF_p$. Therefore, each factor of $\tilde{K}_{4p}$ over $\FF_p$ has the form $c(x^2 - x + 1)^j$, where $c \in \FF_p^\times$ and $1 \le j \le \frac{d_{4p}}{2}$. Therefore, if $\tilde{K}_{4p}$ is reducible over $\QQ$, then $\tilde{K}_{4p}(\omega)$ has to be divisible by $p^2$ over $\ZZ$. We will verify by direct calculations that this is not the case, proving that $\tilde{K}_{4p}$ is irreducible over $\QQ$. 
        
        If $p \equiv 1 \pmod{3}$, then 
        \[
            \tilde{K}_{4p}(x) = T_{4p}(-x) = \frac{x^{4p} + (1 - x)^{4p} + 1}{(x^2 - x + 1)^2} = \frac{K_{4p}(x)}{(x^2 - x + 1)^2},
        \]
        and the Taylor expansion of $K_{4p}$ at $\omega$ has the form $\frac{K''_{4p}(\omega)}{2}(x - \omega)^2 + (x - \omega)^3Q(x)$ for some $Q \in \CC[x]$. Therefore, 
        \begin{align*}
            \tilde{K}_{4p}(x) &= \frac{K''_{4p}(\omega)+ 2(x - \omega)Q(x)}{2(x - 1 + \omega)^2} \\
            &= \frac{4p(4p - 1)(\omega^{4p - 2} + (1 - \omega)^{4p -2}) + 2(x - \omega)Q(x)}{2(x - 1 + \omega)^2} \\
            &= \frac{2p(4p - 1)(\omega^2 + (1 - \omega)^2)+ (x - \omega)Q(x)}{(x - 1 + \omega)^2} \\
            &= \frac{-2p(4p - 1) + (x - \omega)Q(x)}{(x - 1 + \omega)^2},
        \end{align*}
        and hence $\tilde{K}_{4p}(\omega) = -\frac{2p(4p - 1)}{(2\omega - 1)^2} = \frac{2p(4p - 1)}{3}$. If $p \equiv 2 \pmod{3}$, then 
        \[
            \tilde{K}_{4p}(x) = T_{4p}(-x) = \frac{x^{4p} + (1 - x)^{4p} + 1}{x^2 - x + 1} = \frac{K_{4p}(x)}{x^2 - x + 1},
        \]
        and the Taylor expansion of $K_{4p}$ at $\omega$ has the form $K'_{4p}(\omega)(x - \omega) + (x - \omega)^2Q(x)$ for some $Q \in \CC[x]$. Therefore, 
        \begin{align*}
            \tilde{K}_{4p}(x) &= \frac{K'_{4p}(\omega)+ (x - \omega)Q(x)}{x - 1 + \omega} \\
            &= \frac{4p(\omega^{4p - 1} - (1 - \omega)^{4p - 1}) + (x - \omega)Q(x)}{x - 1 + \omega} \\
            &= \frac{4p(2\omega - 1)+ (x - \omega)Q(x)}{x - 1 + \omega},
        \end{align*}
        and hence $\tilde{K}_{4p}(\omega) = 4p$. In either case, $p^2 \nmid \tilde{K}_{4p}(\omega)$, as desired.
    \end{proof}

    The proofs of the last two theorems were based on the facts that $K_2(x) = 2(x^2 - x + 1)$ and $K_4(x) = 2(x^2 - x + 1)^2$. However, we do not have similar nice equalities for polynomials $K_{2m}$ with $m \ge 3$, and the proofs of the irreducibility of polynomials $K_{2lp}$, with $l$ being a fixed integer, become more difficult and require additional conditions. Below, we will try to prove a similar theorem for the polynomials $\tilde{K}_{6p^e}$, which includes the case $\tilde{K}_{6p}$. First, we will need the following lemma:
    
    \begin{lemma} \label{splitting}
        If $K_6(x) = 2x^6 - 6x^5  + 15x^4 - 20x^3 + 15x^2 - 6x + 2 $ has a root modulo an odd prime $p$, then it splits over $\FF_p$. Furthermore, \(H\) acts transitively on the roots of \(K_6\) in $\FF_p$.
    \end{lemma}
    \begin{proof}
        Since $p$ is an odd prime and $K_6(0) = K_6(1) = 2$, $0$ and $1$ are not roots of $K_6$ in $\FF_p$, and the transformations of $H$ are well-defined for the roots of $K_6$ over $\FF_p$. Suppose $p \neq 3, 11$. Note that these transformations form a group isomorphic to $S_3$ that acts on the set of roots of $K_6$ modulo $p$. We claim that this action is free. For this, we have to show that neither of these transformations fix any of the roots of $K_6$ over $\FF_p$. Let $\alpha$ be a root of $K_6$ modulo $p$. 
        \begin{itemize}
            \item If $\alpha = 1 - \alpha$, then $\alpha = \frac{1}{2}$ in $\FF_p$ and $K_6(\alpha) = \frac{33}{32}$. Since $p \neq 3, 11$, this is a contradiction. 

            \item If $\alpha = \frac{1}{\alpha}$, then $\alpha = \pm 1$. However, $K_6(1) = 2$, and $K_6(-1) = 66$, and $p \neq 2, 3, 11$, yielding a contradiction. 

            \item If $\alpha = \frac{\alpha}{\alpha - 1}$, then $\alpha = 2$, and $p \mid K_6(2) = 66$, which is again impossible. 

            \item If $\alpha = \frac{1}{1 - \alpha}$, then $\alpha^2 - \alpha + 1 = 0$, in $\FF_p$. Thus, $\alpha^3 = -1$ and 
            \begin{align*}
                K_6(\alpha) &= \alpha^6 + (1 - \alpha)^6 + 1 \\
                &= 1 + (\alpha^2 - 2\alpha + 1)^3 + 1 \\
                &= 2 + (-\alpha)^3 \\
                &= 3. 
            \end{align*}
            This is again a contradiction since $p \neq 3$. 

            \item If $\alpha = \frac{\alpha - 1}{\alpha}$, then $\alpha^2 - \alpha + 1 = 0$ in $\FF_p$, and we can proceed as in the previous case. 
        \end{itemize}
        
        Thus, a group of order $6$ acts freely on the set of roots of $K_6$ over $\FF_p$. Hence the stabilizers of this action are trivial, and by the Orbit-Stabilizer theorem, the orbits of this action have order $6$. Since $K_6$ is a polynomial of degree $6$ over $\FF_p$, it has at most $6$ roots. Thus, this action is transitive and $K_6$ has exactly $6$ roots. 

        For $p = 3$, note that $K_6(x) = 2(x - 2)^6$ over $\FF_p$ and the claim holds trivially. 

        For $p = 11$, note that $K_6(x) = 2(x - 2)^2(x - 6)^2(x - 10)^2$ over $\FF_p$, so the polynomial splits modulo $p$. It remains to note that $6 = 2^{-1}$ and $10 = 1 - 2$ in $\FF_{11}$, so the action is transitive and any two distinct roots can be obtained from each other by one of the transformations of $H$.
    \end{proof}
    
    The following theorem gives a sufficient condition for the irreducibility of the polynomials $K_{6m}$, when $m$ is a power of a prime. 
    
    \begin{theorem} \label{irreducibility_criterion}
        Let $p$ be an odd prime. Suppose the following conditions hold:
        \begin{enumerate}[label=(\alph*)]
            \item $K_{6}$ has a root over $\FF_p$,
            \item $p^2 \nmid K_{6p}(\alpha)$, for some $\alpha \in \mathbb{Z}$ such that $p \mid K_6(\alpha)$.
        \end{enumerate}
        Then $\tilde{K}_{6p^e}$ is irreducible over $\QQ$ for each positive integer $e$. 
    \end{theorem}
    \begin{proof}
        Note that $\tilde{K}_{6p^e} = K_{6p^e}$. It is well known that for any polynomial $f \in \ZZ[x]$, $f(x + p) - f(x) = pf'(x) \pmod{p^2}$. Since $K'_{6p^e}(x) = 0$ over $\FF_p$, it follows that for an integer $a \in \ZZ$, the residue of $K_{6p^e}(a)$ modulo $p^2$ depends only on the residue of $a$ modulo $p$. Note that for any $a \in \ZZ$, $K_{6p^e}(a) \equiv K_{6p}(a) \pmod{p^2}$. 

        Now let $\alpha$ be an integer such that $p \mid K_6(\alpha)$. Then $K_6$ has a root over $\FF_p$, and hence by Lemma \ref{splitting}, $K_6$ splits over $\FF_p$, and $H$ acts transitively on the set of roots of $K_6$ over $\FF_p$. 
        
        Suppose $K_6(x) = 2 \prod_{j = 1}^6(x - \alpha_j)$ over $\FF_p$, where $\alpha = \alpha_1$. Then for a fixed positive integer $e$ 
        \begin{align*}
            K_{6p^e}(x) &= x^{6p^e} + (1 - x)^{6p^e} + 1 \\
            &= (x^6 + (1 - x)^6 +  1)^{p^e} \\
            &= K_6(x)^{p^e} \\
            &= 2^{p^e} \prod_{j = 1}^6 (x - \alpha_j)^{p^e} \\
            &= 2 \prod_{j = 1}^6 (x - \alpha_j)^{p^e}
        \end{align*}
        over $\FF_p$. Suppose $K_{6p^e}$ is reducible over $\QQ$. Then write $K_{6p^e} = f_1(x)  \dotsm f_s(x)$, where $f_1, \ldots, f_s \in \ZZ[x]$ are irreducible and have positive leading coefficients. Then
        \begin{align*}
            2\prod_{j = 1}^6 (x - \alpha_j)^{p^e} = f_1(x) \dotsm f_s(x).
        \end{align*}
        over $\FF_p$. From the condition (b), it follows that there is a unique index $i \in \{1, 2, \ldots, s\}$ such that $p \mid f_i(\alpha)$. Without loss of generality, assume that $i = 1$. Since \(f_j(\alpha) \not \equiv 0\pmod p\) for indices \(j>1\), the multiplicities of $\alpha$ in $K_{6p^e}$ and $f_1$ are equal over $\FF_p$. By Theorem \ref{factorsymmetries_intro}, $f_1$ satisfies $f_1(x) = f_1(1 - x) = f_1^*(x)$. Thus, it satisfies the same equalities in $\FF_p[x]$. Since $H$ acts transitively on the set of roots of $K_6$, all the roots of $K_6$ over $\FF_p$ are roots of $f_1$ over $\FF_p$ as well. Furthermore, since $f_1(x) = f_1(1 - x) = f_1^*(x)$, the multiplicities of the roots of $f_1$ are the same. Therefore $K_{6p^e} = f_1$ over $\FF_p$. But then $s = 1$, and $K_{6p^e}$ is irreducible. 
    \end{proof}

    \begin{remark}
        The irreducibility of $K_{6p^e}$, for $p = 2$ and $p = 3$, follows from Propositions \ref{p=2} and \ref{p=3}, respectively.
    \end{remark}

    \begin{example}
        In the proof of Lemma \ref{splitting} it was noted that $2$ is a root of $K_6$ modulo $11$. On the other hand, $K_{66}(2) = 73786976294838206466$, which is not divisible by $11^2$. Therefore, by Theorem \ref{irreducibility_criterion}, $K_{6 \cdot 11^e}$ is irreducible over $\QQ$ for each positive integer $e$. 
    \end{example}

    \begin{example}
        Note that $4$ is a root of $K_6$ modulo $19$. Unfortunately, $K_{114}(4)$ is divisible by $19^2$, and hence Theorem \ref{irreducibility_criterion} is not applicable in this case. 
    \end{example}
    
    \begin{remark}
        Computations with SageMath show that the only odd prime $p < 10000$ modulo which $K_6$ has a root, but the condition (b) of Theorem \ref{irreducibility_criterion} is not satisfied, is $19$. It is natural to ask whether $p = 19$ is the only such prime. Unfortunately, this assertion might be very difficult to prove and we do not have any results on this. Here is a list of all primes up to \(10000\) for which the conditions of Theorem \ref{irreducibility_criterion} are satisfied:
        \begin{center}
            3, 11, 71, 127, 149, 151, 173, 233, 283, 293, 313, 383, 397, 419, 443, 461, 569, 607, 647, 719, 761, 787, 947, 971, 983, 1051, 1213, 1231, 1237, 1321, 1327, 1361, 1367, 1439, 1453, 1481, 1499, 1511, 1549, 1553, 1601, 1741, 1759, 1889, 1949, 1999, 2003, 2027, 2029, 2251, 2267, 2287, 2393, 2399, 2423, 2441, 2551, 2557, 2647, 2677, 2683, 2689, 2711, 2741, 2797, 2843, 3001, 3037, 3079, 3307, 3433, 3449, 3457, 3491, 3559, 3571, 3581, 3593, 3697, 3761, 3797, 3907, 3967, 4001, 4003, 4079, 4099, 4133, 4139, 4273, 4289, 4397, 4457, 4567, 4637, 4639, 4643, 4789, 4801, 4817, 4831, 4909, 4943, 5003, 5011, 5023, 5113, 5197, 5281, 5297, 5303, 5351, 5407, 5413, 5477, 5573, 5623, 5849, 5879, 5927, 6037, 6073, 6089, 6091, 6121, 6229, 6379, 6619, 6719, 6761, 6779, 6791, 6833, 6883, 6907, 6961, 6983, 7151, 7187, 7229, 7297, 7307, 7411, 7451, 7457, 7489, 7541, 7547, 7561, 7573, 7589, 7621, 7673, 7681, 7757, 7853, 7867, 7949, 8101, 8111, 8117, 8191, 8209, 8231, 8233, 8243, 8311, 8443, 8527, 8581, 8623, 8681, 8707, 8731, 8761, 8863, 8867, 8963, 9103, 9109, 9127, 9133, 9137, 9187, 9241, 9391, 9397, 9437, 9521, 9533, 9623, 9791, 9811, 9887, 9901, 9907, 9923, 9941
        \end{center}
    \end{remark}

    \section{The Galois group and the discriminant} \label{discriminant_section}
    In this section, we investigate the Galois group and the discriminant of the polynomials $\tilde{K}_{2m}$. For a polynomial $h \in \CC[x]$, denote its discriminant by $\disc(h)$. 
    
    \begin{lemma} \label{Discriminant_Formula}
        Denote $\zeta = \zeta_{2m - 1} = e^{\frac{2i\pi}{2m - 1}}$. Then the following formula holds:
        \[
            \disc(K_{2m}) = (-1)^m(2m)^{2m}(2^{2m - 1} + 1) \left( \prod_{j = 1}^{m - 1} \Big(1 + (1 + \zeta^j)^{2m - 1}\Big)\right)^2.
        \]
    \end{lemma}
    \begin{proof}
        First, we claim that the numbers $\frac{\zeta^j}{1 + \zeta^j}$, $j = 1, 2, \ldots, 2m - 1$ are the roots of $K'_{2m}$. Since they are all distinct and $\deg{K'_{2m}} = 2m-1$, it suffices to show that these numbers are roots of $K'_{2m}$. Note that $K'_{2m}(x) = 2m \left(x^{2m - 1} - (1 - x)^{2m - 1} \right)$, so
        \[
            K'_{2m} \left( \frac{\zeta^j}{1 + \zeta^j } \right) = 2m \left( \left( \frac{\zeta^j}{1 + \zeta^j}\right)^{2m - 1} -  \left( \frac{1}{1 + \zeta^j}\right)^{2m - 1} \right) = 0.
        \]
        Note that the leading coefficients of $K_{2m}$ and $K'_{2m}$ are $2$ and $4m$, respectively. Therefore, 
        \begin{align*}
            \disc(K_{2m}) &= \frac{(-1)^{\frac{2m(2m -1)}{2}}}{2} \res(K_{2m}, K'_{2m}) \\
            &= \frac{(-1)^m (4m)^{2m}}{2} \prod_{j = 1}^{2m - 1} K_{2m} \left( \frac{\zeta^j}{1 + \zeta^j}\right) \\
            &= \frac{(-1)^m (4m)^{2m}}{2} \prod_{j = 1}^{2m - 1} \left(\left( \frac{\zeta^j}{1 + \zeta^j}\right)^{2m} + \left( \frac{1}{1 + \zeta^j} \right)^{2m} + 1 \right) \\
            &= \frac{(-1)^m (4m)^{2m}}{2} \cdot \frac{\prod_{j = 1}^{2m - 1} \left( \zeta^j + 1 + (1 + \zeta^j)^{2m} \right)}{\left( \prod_{j = 1}^{2m - 1}\left( 1 + \zeta^j \right)\right)^{2m}} \\
            &= \frac{(-1)^m (4m)^{2m}}{2} \cdot \frac{\prod_{j = 1}^{2m - 1} \left( 1 + (1 + \zeta^j)^{2m - 1} \right)}{\left( \prod_{j = 1}^{2m - 1}\left( 1 + \zeta^j \right)\right)^{2m - 1}}.
        \end{align*}
        Since $\zeta, \zeta^2, \ldots, \zeta^{2m - 1}$ are the roots of $g(x) = x^{2m - 1} - 1$, $g(x) = \prod_{j = 1}^{2m - 1} (x - \zeta^j)$, and hence $2 = -g(-1) = \prod_{j =1}^{2m - 1} (1 + \zeta^j)$. It follows that
        \begin{align*}
            \disc(K_{2m}) &= \frac{(-1)^m (4m)^{2m}}{2} \cdot \frac{\prod_{j = 1}^{2m - 1}  \left( 1 + (1 + \zeta^j)^{2m - 1} \right)}{2^{2m - 1}} \\
            &= (-1)^m(2m)^{2m} \prod_{j = 1}^{2m - 1} \left( 1 + (1 + \zeta^j)^{2m - 1} \right) \\
            &= (-1)^m (2m)^{2m} (1 + 2^{2m - 1}) \prod_{j = 1}^{2m - 2} \left( 1 + (1 + \zeta^j)^{2m - 1} \right).
        \end{align*}
        Finally, note that for $j = 1, 2, \ldots, m - 1$, $(1 + \zeta^j)^{2m - 1} = (1 + \zeta^{2m - 1 - j})^{2m - 1}$, and hence
        \begin{align*}
            \disc(K_{2m}) &= (-1)^m(2m)^{2m}(2^{2m - 1} + 1) \left( \prod_{j = 1}^{m - 1} \Big(1 + (1 + \zeta^j)^{2m - 1}\Big)\right)^2. \qedhere
        \end{align*}
    \end{proof}

    Since $\tilde{K}_{6m} = K_{6m}$, Lemma \ref{Discriminant_Formula} allows to compute the discriminant of $\tilde{K}_{6m}$. The following helps to compute the discriminant of $\tilde{K}_{6m + 2}$:

    \begin{corollary} 
        Denote $\xi = e^{\frac{2i\pi}{6m + 1}}$. Then the following formula holds:
        \[
            \disc(\tilde{K}_{6m + 2}) = \frac{(-1)^m(6m + 2)^{6m - 2} (2^{6m + 1} + 1)}{3(\cont(K_{6m + 2}))^{6m}} \left( \prod_{j = 1}^{3m} \left(1 + (1 + \xi^j)^{6m + 1} \right)\right)^2. 
        \]
    \end{corollary}
    \begin{proof}
        Note that 
        \[
            K_{6m + 2}(x) = \cont(K_{6m + 2}) \tilde{K}_{6m + 2}(x) (x^2 - x + 1). 
        \]
        Therefore, denoting $P(x) \colonequals \frac{K_{6m + 2}(x)}{x^2 - x + 1} = \cont(K_{6m + 2})\tilde{K}_{6m + 2}(x)$, we obtain $K_{6m + 2}(x) = P(x)(x^2 - x + 1)$, and hence
        \[
            \disc(K_{6m + 2}) = \disc ( P(x) )  \disc(x^2 - x + 1) 
            \res ( P(x), x^2 - x + 1 )^2. 
        \]
       Note that $\disc(x^2 - x + 1) = -3$, and
       \[
           \res(P(x), x^2 - x + 1) = P(\omega)P(\overline{\omega}) = P(\omega)P(1 - \omega) = P(\omega)^2.
       \]
       On the other hand, 
       \[
            K_{6m + 2}'(x) = (6m + 2)\left( x^{6m + 1} + (x - 1)^{6m + 1} \right) = P'(x)(x^2 - x + 1) + P(x)(2x - 1),
       \]
       and hence
       \[
            (2\omega - 1)P(\omega) = (6m + 2) \left( \omega^{6m + 1} + (1 - \omega)^{6m + 1}\right) = (6m + 2)\left( \omega + \omega^2 \right) = (6m + 2)(2\omega - 1).  
       \]
       So, $P(\omega) = 6m + 2$ and $\res(P(x), x^2 - x + 1) = (6m + 2)^2$. Therefore, 
       \[
            \disc(P) = -\frac{\disc(K_{6m + 2})}{3(6m + 2)^4}.
       \]
       Since $P = \cont(K_{6m + 2})\tilde{K}_{6m + 2}$ and $\deg{P} = 6m$, $\disc(P) = (\cont(K_{6m + 2}))^{6m} \disc (\tilde{K}_{6m + 2})$. Combining this with Lemma \ref{Discriminant_Formula}, we obtain
       \[
            \disc(\tilde{K}_{6m + 2}) = \frac{(-1)^m(6m + 2)^{6m - 2} (2^{6m + 1} + 1)}{3(\cont(K_{6m + 2}))^{6m}} \left( \prod_{j = 1}^{3m} \left(1 + (1 + \xi^j)^{6m + 1} \right)\right)^2. \qedhere
       \]
    \end{proof}

    \begin{remark}
        This method does not work for the polynomials $\tilde{K}_{6m + 4}$ as the polynomials $K_{6m + 4}$ are not square-free, and hence Lemma \ref{Discriminant_Formula} only yields $\disc(K_{6m + 4}) = 0$. 
    \end{remark}

    \begin{corollary} \leavevmode \label{Quadratic_Subfield}
        \begin{itemize}
            \item The splitting field of $\tilde{K}_{6m}$ over $\QQ$ contains $\sqrt{(-1)^m(2^{6m - 1} + 1)}$.
            \item The splitting field of $\tilde{K}_{6m + 2}$ over $\QQ$ contains $\sqrt{3(-1)^m(2^{6m + 1} + 1)}$.
        \end{itemize}
    \end{corollary}
    \begin{proof}
        We will only prove the first assertion, the second one can be proved similarly. Having Lemma \ref{Discriminant_Formula}, it suffices to show that $S = \prod_{j = 1}^{3m - 1} \Big(1 + (1 + \zeta^j)^{6m - 1}\Big)$ is an integer. It is clear that $S$ is an algebraic integer, hence it suffices to show that $S$ is rational. Note that $S$ belongs to the cyclotomic field $\QQ(\zeta)$. For $u \in \Big( \ZZ/(6m - 1) \Big)^{\times}$ (the group of units), let $\sigma_u$ denote the automorphism of $\QQ(\zeta)$ sending $\zeta \rightarrow \zeta^u$. It is well known that these are all the possible automorphisms of $\QQ(\zeta)$. Since $\QQ(\zeta)/\QQ$ is a Galois extension, showing that $S$ is rational is equivalent to showing that $\sigma_u(S) = S$ for each $u \in \Big( \ZZ/(6m - 1)\Big)^\times$ (by the fundamental theorem of Galois theory). Note that for a fixed $u \in \Big( \ZZ/(6m - 1)\Big)^\times$, we have
        \[
            \sigma_u(S) = \prod_{j = 1}^{3m - 1} \Big(1 + (1 + \zeta^{uj})^{6m - 1}\Big). 
        \]
        Note that the numbers $u, 2u, \ldots, (3m - 1)u$ are distinct in $\ZZ/(6m - 1)$. Furthermore, for each index $j \in \{1, 2, \ldots, 3m - 1\}$, exactly one number among $u, 2u, \ldots, (3m - 1)u$ belongs to the pair $(j, -j)$. Since $(1 + \zeta^{j})^{6m - 1} = (1 + \zeta^{6m - 1 - j})^{6m - 1}$, it follows that 
        \begin{align*}
            \sigma_u(S) &= \prod_{j = 1}^{3m - 1} \Big(1 + (1 + \zeta^{uj})^{6m - 1}\Big) \\
            &= \prod_{j = 1}^{3m - 1} \Big(1 + (1 + \zeta^j)^{6m - 1}\Big) \\
            &= S. \qedhere
        \end{align*}
    \end{proof}

    \begin{proof}[Proof of Theorem \ref{thm:oddpermut}]
        If $n \equiv 10 \pmod{12}$, then $\deg{\tilde{K}_n} \equiv 2 \pmod{4}$. Thus, $\disc(\tilde{K}_{12m + 10})$ is negative, and hence the Galois group of $\tilde{K}_{12m + 10}$ over $\QQ$ contains an odd permutation.
    
        Thus, it remains to prove the cases $n = 6m$ and $n = 6m + 2$, where $m$ is a positive integer. This is equivalent to showing that $\disc({\tilde{K}_{6m}})$ and $\disc({\tilde{K}_{6m + 2}})$ are not perfect squares. Having Corollary \ref{Quadratic_Subfield}, it suffices to show that $2^{12m - 1} + 1$ and $3(2^{12m + 1} + 1)$ are not perfect squares.
        
        For $2^{12m - 1} + 1$, assume by contradiction that there is some integer $y \in \ZZ$ such that $2^{12m - 1} + 1 = y^2$. Then $2^{12m - 1} = (y - 1)(y + 1)$, and hence both $y - 1$ and $y + 1$ are powers of $2$. This occurs only when $y = 3$, but in that case $2^{12m - 1} + 1 = 9$ and $12m - 1 = 3$, which is a contradiction.  

        For $3(2^{12m + 1} + 1)$, simply observe that it is congruent to $3$ modulo $4$, and hence is not the square of an integer. 
    \end{proof}

    Recall that $b_n \colonequals \frac{\deg \tilde{K}_n}{6}$. The following theorem is a generalization of a result of Helou (see \cite{helou}) from the polynomials $E_n$ with odd $n$ to all the polynomials $\tilde{K}_n$.
    
    \begin{proposition}
        The Galois group of $\tilde{K}_n$ over $\QQ$ is isomorphic to an extension of a subgroup of $S_3^{b_n}$ by a subgroup of $S_{b_n}$, and hence the order of the Galois group of $\tilde{K}_n$ over $\QQ$ is less than or equal to $6^{b_n} \cdot b_n !$.
    \end{proposition}
    \begin{proof}
        Let $G$ denote the Galois group of $\tilde{K}_n$ over $\QQ$. Note that the roots of $\tilde{K}_n$ can be partitioned into $b_n$ $6$-tuples
        \begin{gather*}
            \left\{ \alpha_1, 1 - \alpha_1, \frac{1}{\alpha_1}, \frac{\alpha_1}{\alpha_1 - 1}, \frac{1}{1 - \alpha_1}, \frac{1 - \alpha_1}{\alpha_1} \right\}, \\
            \left\{ \alpha_2, 1 - \alpha_2, \frac{1}{\alpha_2}, \frac{\alpha_2}{\alpha_2 - 1}, \frac{1}{1 - \alpha_2}, \frac{1 - \alpha_2}{\alpha_2} \right\}, \\
            \dotsm \\
            \left\{ \alpha_{b_n}, 1 - \alpha_{b_n}, \frac{1}{\alpha_{b_n}}, \frac{\alpha_{b_n}}{\alpha_{b_n} - 1}, \frac{1}{1 - \alpha_{b_n}}, \frac{1 - \alpha_{b_n}}{\alpha_{b_n}} \right\}.
        \end{gather*}
        Let $\Omega$ denote the set of these $6$-tuples. Note that $G$ acts on $\Omega$. Since $\Omega$ has $b_n$ elements, this gives a homomorphism $\varphi: G \rightarrow S_{b_n}$. Then, by the first homomorphism theorem, 
        \[
            \im{\varphi} \cong G / \ker{\varphi}.
        \]
        On the other hand, $\im{\varphi}$ is a subgroup of $S_{b_n}$ and it is not difficult to see that $\ker{\varphi}$ can be naturally embedded into $S_3^{b_n}$. Therefore, $G$ is an extension of a subgroup of $S_3^{b_n}$ by a subgroup of $S_{b_n}$ and $|G| \le |S_3^{b_n}| \cdot|S_{b_n}| = 6^{b_n} \cdot b_n!$.
    \end{proof}
	\section*{Acknowledgments}
	
	We express our sincere gratitude to Mihran Papikian, whose comments were of great help to us as novice researchers, and to Anush Tserunyan for useful remarks about the introduction. We also thank the anonymous referee for a careful reading of the manuscript and for many valuable suggestions. 
    
    \bibliographystyle{amsalpha}
    \bibliography{Bibliography.bib}

\end{document}